\documentclass[a4paper,11pt,reqno]{amsart}
\usepackage[T2A]{fontenc}
\usepackage[cp1251]{inputenc}
\usepackage[english]{babel}
\usepackage{amsmath,amsthm,amsfonts,amssymb}
\usepackage{enumerate}
\theoremstyle{plain}
\newtheorem{theorem}{Theorem}
\newtheorem{lemma}{Lemma}

\newtheorem*{corollary*}{Наслідок}
\newtheorem{proposition}{Proposition}
\theoremstyle{definition}

\theoremstyle{remark}
\newtheorem{remark}{Remark}
\newtheorem*{remark*}{Remark}


\uccode`ґ=`Ґ \uccode`Ґ=`Ґ \lccode`Ґ=`ґ \lccode`ґ=`ґ
\uccode`є=`Є \uccode`Є=`Є \lccode`Є=`є \lccode`є=`є
\uccode`і=`І \uccode`І=`І \lccode`І=`і \lccode`і=`і
\uccode`ї=`Ї \uccode`Ї=`Ї \lccode`Ї=`ї \lccode`ї=`ї

\begin{document}

\title[On the Hausdorff dimension faithfulness and the Cantor series expansion]
{On the Hausdorff dimension faithfulness  \\and the Cantor series expansion}
\author[S. Albeverio, G.Ivanenko,  M.Lebid, G. Torbin  ]{Sergio Albeverio$^{1,2,3,4,5}$,
Ganna Ivanenko$^{6}$,\\ Mykola Lebid$^{7,8}$, Grygoriy Torbin$^{9,10}$}
\date{}
\maketitle
\begin{abstract}
 We study  families $\Phi$ of coverings which are faithful for the Hausdorff dimension calculation on a given set $E$  (i. e., special relatively narrow families of coverings leading to the classical Hausdorff dimension of an arbitrary subset of $E$)  and which  are natural generalizations of  comparable net-coverings. They are shown to be very useful for the determination or estimation of the Hausdorff dimension of sets and probability measures.
 We give general necessary and sufficient conditions  for a covering family to be faithful and new techniques for proving faithfulness/non-faithfulness for the family of cylinders generated by expansions of real numbers. Motivated by applications in the multifractal analysis of infinite Bernoulli convolutions, we study in details the Cantor series expansion and prove necessary and sufficient conditions for the corresponding net-coverings to be faithful. To the best of our knowledge this  is the first known sharp   condition of the faithfulness for a class of covering families containing both faithful and non-faithful ones.

  Applying our results, we characterize fine fractal properties of probability measures with independent digits of the Cantor series expansion and show that a class  of faithful net-coverings essentially wider that the class of comparable ones.   We construct, in particular,  rather simple examples of faithful families $\mathcal{A}$ of net-coverings  which are "extremely non-comparable" to the Hausdorff measure.
\end{abstract}

$^1$~Institut f\"ur Angewandte Mathematik, Universit\"at Bonn,
Endenicher Allee 60, D-53115 Bonn (Germany); $^2$HCM and ~SFB 611, Bonn; $^3$ BiBoS,
Bielefeld--Bonn; $^4$~CERFIM, Locarno; $^5$~IZKS, Bonn; E-mail:
albeverio@uni-bonn.de

$^{6}$~National Dragomanov Pedagogical University, Pyrogova str. 9, 01030 Kyiv
(Ukraine); E-mail: anna.ivanenko@ukr.net

$^7$ Fakult\"{a}t  f\"{u}r Mathematik, Universit\"{a}t Bielefeld, Postfach 10 01 31, D-33501,  Bielefeld (Germany); $^{8}$~National Dragomanov Pedagogical University, Pyrogova str. 9, 01030 Kyiv
(Ukraine); E-mail: mlebid@math.uni-bielefeld.de

$^{9}$~ National Dragomanov Pedagogical University, Pyrogova str. 9, 01030 Kyiv
(Ukraine)~$^{10}$Institute for Mathematics of NASU, Tereshchenkivs'ka str. 3, 01601 Kyiv (Ukraine); E-mail:
torbin@iam.uni-bonn.de (corresponding author)
\medskip

\medskip


\section{Introduction}

The notion of the Hausdorff dimension  is well-known now  and is of great importance  in mathematics as well as in
diverse applied problems (see, e.g., \cite{Fal, Mandelbrot83, Pesin97, Triebel}).
In many situations the determination (or even estimations) of this dimension for sets from
a given family or even for a given set is a rather non-trivial problem (see, e. g. , \cite{Baranski, Bil1, Fal} and references therein).
Different approaches and special methods for the determination of the Hausdorff dimension are collected in \cite{Fal, Fal97, Mattila}.
A new approach based on the theory of transformations preserving the Hausdorff dimension (DP-transformations)  was
presented in \cite{APT, APT2}. In this paper  we develop another approach which is deeply connected with the theory of
DP-transformations as well as with the following well known approach:
to simplify the calculation of the Hausdorff dimension of a given set it is extremely useful to have an appropriate and a relatively narrow family of admissible coverings which lead to the same value of the dimension. We shall start to deal with one-dimensional sets and show later how our results can be naturally extended to the multidiminsional case and to the general case of metric spaces.

Without loss of generality we shall consider subsets from the unit interval. Let  $ \Phi$ be a \textit{fine family of coverings} on $[0,1]$, i.e., a family  of subsets of $[0,1]$ such that for any  $\varepsilon >0$ there exists an at most countable $\varepsilon$ - covering $\{E_{j}\}$
of $[0,1]$ with $E_j \in \Phi$.
Let us shortly recall that \textit{the $\alpha $-dimensional Hausdorff measure}
of a set $E\subset [0,1]$ w. r. t.  a given  fine family of coverings $ \Phi$  is defined by
\[H^{\alpha } (E,  \Phi)=\mathop{\lim }\limits_{\varepsilon \to \infty } \left[\mathop{\mathop{\inf }\limits_{\left|E_{j} \right|\le \varepsilon } \left\{\sum _{j}\left|E_{j} \right|^{\alpha }  \right\}}\limits_{} \right]=\mathop{\lim }\limits_{\varepsilon \to \infty } H_{\varepsilon }^{\alpha } (E,\, \Phi ),\]
where the infimum is taken over all at most countable $\varepsilon $-coverings $\{ E_{j} \} $ of $E$, $E_{j} \in \Phi$.

We remark that, generally speaking, $H^{\alpha }(E,  \Phi)$ depends on the family $\Phi$. The family of all subsets of $[0, 1]$ and the family of all closed (open) subintervals of $[0, 1]$  give rise
to the same $\alpha$-dimensional Hausdorff measure, which will be denoted by
$H^{\alpha } (E)$.
 The nonnegative number
\[\dim _{H} (E,\, \Phi)=\inf \{ \alpha :\, \, H^{\alpha } (E,\, \Phi )=0\}\]
\noindent is called the Hausdorff dimension of the set $E\subset [0,1]$ w. r. t. a family $\Phi$.
If $\Phi$ is the family of all subsets of $[0, 1]$, or $\Phi$ coincides with the family of
all closed (open) subintervals of [0,1], then $\dim _{H} (E,\, \Phi)$ equals to the classical Hausdorff dimension $\dim _{H} (E)$ of the subset $E \subset [0,1]$.

The  notion of comparable net measures are also well known now (see, e. g. , \cite{Falconer85, R}). Roughly speaking, net measures are special cases of $H^{\alpha }(E,  \Phi)$, where the family $\Phi$ consists of sets with the following properties: 1) if $A_1$ and $A_2$ belong to $\Phi$, then $A_1 \subset A_2$ or  $A_2 \subset A_1$ or $A_1 \bigcap A_2 = \emptyset$; 2) $\Phi$ is countable; 3) at most a finite number of sets from $\Phi$ contain any given set from $\Phi$. Then the corresponding net measure $H^{\alpha } (E,  \Phi)$ is said to be comparable to Hausdorff measure if the ratios of measures are bounded above and below.  It has been shown that comparable net measures are very useful in the study of Hausdorff measures (see, e.g., \cite{Besicovitch, Falconer85, Marstrand1954, R} and references therein).

In this paper we actually develop theory of measures  which are generalizations of comparable net measures in the following sense.

\noindent \textbf{Definition}. A fine covering family $\Phi$ is said to be \textit{faithful family of coverings}~\textit{(non-faithful family of coverings)} for the Hausdorff dimension calculation on $[0,1]$ if $$\dim _{H} (E,\Phi)=\dim _{H} (E), ~~~\forall E\subseteq [0,1]$$ $$(\mbox{resp.} ~~\exists E\subseteq [0,1]: \dim _{H} (E,\Phi)\neq\dim _{H} (E)).$$

It is clear that any family $\Phi$ of  comparable net-coverings (i.e., net-coverings which generate comparable net-measures)  is faithful.
Conditions for a fine covering family to be faithful were studied by many authors (see, e.g., \cite{AT2, B, Cutler, PT_NZ2003} and references therein). First steps in this direction have been done by A. Besicovitch (\cite{Besicovitch}), who proved the faithfulness for the family of cylinders of binary expansion. His result was extended by P. Billingsley (\cite{B}) to the family of $s$-adic cylinders, by M. Pratsiovytyi (\cite{TuP}) to the family of $Q$-$S$-cylinders, and by S. Albeverio and G. Torbin (\cite{AT2}) to  the family of $Q^*$-cylinders for those matrices $Q^*$ whose elements $p_{0k}, p_{(s-1)k}$ are bounded from zero.  Some general sufficient conditions for the  faithfulness of a given  family of  coverings  are also  known (\cite{Cutler}, \cite{PT_NZ2003}). Let us mentioned here that all these results were obtained by using the standard approach: if for a given family $\Phi$ there exist  positive constants $\beta \in \mathbb{R}$ and $N^{*}\in \mathbb{N}$ such that
for any interval $B=(a,b)$ there exist at most $N^{*}$ sets $B_j \in \Phi$ which cover $(a,b)$ and $|B_j|\leq \beta\cdot |B|$, then the family $\Phi$ is faithful. It is clear that all above mentioned families of net-coverings are even comparable.

  It is rather paradoxical that  initial  examples of non-faithful families of coverings appeared firstly  in two-dimensional case (as a result of active studies of self-affine sets during the last decade of XX century (see, e.g., \cite{BernardiBondioli})). The family of cylinders of the classical continued fraction expansion can probably be considered as the first (and rather unexpected) example of non-faithful one-dimensional net-family of coverings (\cite{PerTor}). By using approach, which has been invented by Yuval Peres to prove non-faithfulness of the family of continued fraction cylinders (\cite{PerTor}), in  \cite{AlbKonNikTor} authors have proven the non-faithfulness for the family of cylinders of  $Q_\infty$-expansion with polynomially decreasing elements $\{q_i\}$. The latter two  families of  coverings give  examples of non-comparable net measures. So, it is natural to ask about the existence of faithful covering families which are not comparable.

  We study this problem and give general  necessary and sufficient conditions for a fine covering system to be faithful. The main aim of the paper is to study faithful properties  of the covering families which are generated by the famous Cantor series expansions.
   Let us recall that for a given sequence $\left\{n_{k} \right\}_{k=1}^{\infty } $ with $n_{k} \in \mathbb{N} \backslash \{ 1\} ,\, k \in
\mathbb{N}$ the  expression of $x \in [0,1]$ in the following form
\[x=\sum_{k=1}^{\infty}\frac{\alpha_{k}}{n_{1} \cdot n_{2} \cdot \ldots \cdot n_{k}}=:\Delta_{\alpha_{1}\alpha_{2}...\alpha_{k}...},~\alpha_{k} \in \{0,~1,~...,~n_{k}-1\}\]
is said to be the Cantor series expansion of $x$.  These expansions, which have been initially studied by G. Cantor in 1869 (see., e.g. \cite{Cantor1869}), are natural generalizations of the classical $s$-adic expansion for reals. Cantor series expansions have been intensively studied from different points of view during last century (see, e.g., \cite{ManceDiss, Schweiger} and references therein). Our own motivations to study faithful properties of such expansions came from our investigations on fine fractal properties of infinite Bernoulli convolutions, i.e., probability distributions of the following random variables
\begin{equation} \label{1}
\xi  =\sum _{k=1}^{\infty }\xi_{k} a_{k} ,
\end{equation}
 where $\sum\limits _{k=1}^{\infty }a_{k}$ is a convergent positive series, and $\xi_{k} $ are independent random variables taking values  0 and 1 with probabilities $p_{0k} $ and $p_{1k} $ respectively. Measures of this form have been studied since
1930's from the pure probabilistic point of view as well as for
their applications in harmonic analysis, dynamical systems and
fractal analysis \cite{Peres}.  The Lebesgue structure  and fine fractal properties of the distribution of $\xi$ are well studied for the case where $r_k:=\sum\limits_{i=k+1}^{\infty}a_i \geq a_k$ for all large enough $k$ (see, e.g., \cite{AlbeverioTorbin2007,Cooper}).  The case where $a_k< r_k$ holds for an infinite number of $k$ can be considered  as a <<Terra incognita>> in this field. Even for the case where  $a_k=\lambda^k$ and $p_{0k}=\frac 12$ the  problem of singularity is still open (\cite{Solomyak95}). Main problems here are related to the fact that almost all points from the spectrum of $\xi$ have uncountably many different expansion in the form $\sum \varepsilon_k a_k, \varepsilon_k \in \{0,1\}.$ This is the so-called <<Bernoulli convolutions with large overlaps>>.  We consider two special classes of such measures.
 The first one is generated by sequences $a_k$ with the following properties:  $ \exists \left\{m_k\right\}$ such that $r_{m_k} > a_{m_k} = a_{m_k+1}+...+ a_{m_k+s_k},$ $s_k= m_{k+1}-m_k-1 \ge 2, ~~ k \in  N, $ $r_j=a_j,$ for $j \notin \{ m_k \}.$

 The second one is connected to the sequences $a_k$ such that  $ \forall k\in N$  $~ \exists s_{k} \in N\bigcup\{0\}: $ $~~a_{k}  = a_{k+1}  = ... = a_{k+s_{k} }  =  r_{k+s_{k} }$, and  $s_k>0$  for an infinite number of indices $k$.
 In both cases singularity plays a generic role and to study fine fractal properties of the corresponding probability distributions it is necessary to have knowledge on faithfulness (non-faithfulness) of fine families of partitions which are turned to be  the  Cantor series partitions.

 Main result of the present paper  states that  the family $\mathcal{A}$ of Cantor coverings of the unit interval is faithful for the Hausdorff dimension calculation  if and only if
\begin{equation}
\mathop{\lim }\limits_{k\to \infty } \frac{\ln n_{k} }{\ln n_{1} \cdot n_{2} \cdot \ldots \cdot n_{k-1} } =0.
\end{equation}

To the best of our knowledge this theorem gives the first necessary and sufficient  condition of the faithfulness for a class of covering families containing both faithful and non-faithful ones.
The proof of this result is given in the next Section. As a corollary of our results we characterize fine fractal properties of probability measures with independent digits of the Cantor series expansion and show that a class  of faithful net-coverings essentially wider that the class of comparable net-coverings.   We construct, in particular, simple examples of faithful families $\mathcal{A}$ of net-coverings  which are "extremely" non-comparable to the Hausdorff measure.

\section{Sharp conditions for the Hausdorff dimension faithfulness of\\ the Cantor series expansion}
In this Section we give some general conditions for a fine covering family to be faithful for the Hausdorff dimension calculation  and prove necessary and sufficient conditions for the Cantor series net-coverings to be faithful.


We start firstly with a very useful lemma, which can be proven easily, and, nevertheless,  presents general necessary and sufficient conditions for the faithfulness.

\begin{lemma}\label{lem_dim} Let $\Phi $ be a fine covering family on $[0,1]$. Then  $\Phi$ is faithful on the unit interval if and only if there exists a positive constant $C$ such that for any $ E\subset [0,1]$, any $\alpha \in (0,1]$ and any  $\delta \in (0,\alpha )$ the following inequality holds:
\begin{equation} \label{Lemma_1*}
 H^{\alpha } (E,\Phi )\le C \cdot  H^{\alpha -\delta } (E).
\end{equation}
\end{lemma}

Let us mention that this lemma can be obviously generalized to a multidimensional Euclidean space and even to any metric space, which can be equipped by fine covering families.

Based on the latter lemma one can easily get the following sufficient condition for the faithfulness of a fine covering family.

\begin{lemma}\label{lem_trust}

Let $\Phi$ be a fine covering family of $[0,1]$. Assume that there exists a positive
constant $C$ and a function $N(x):R_{+} \rightarrow \mathbb{N}$ such that:

\noindent 1) for any interval $I \subset[0,1]$ there exist at most  $N(\left|I\right|)$ subsets \[\triangle^{I}_{1},~
\triangle^{I}_{2},~...,~\triangle^{I}_{l(I)} \in \Phi\]  with \[l(I)\leq N(\left|I\right|),~ |\triangle^{I}_{j}|\leq\left|I\right|~ \textrm{and} ~ I
\subset\bigcup^{l(I)}_{j=1}\triangle^{I}_{j};\]

\noindent 2) for any $\delta >0$ there exists $\varepsilon_{1}(\delta) >0$ such that
\[N(\left|I\right|)\cdot
\left|I\right|^{\delta } \le C,~\forall I \subset [0,1]~ \text{with}~ |I|<\varepsilon_{1}(\delta).\]

\noindent Then the family $\Phi $ is faithful on $[0,1]$. 

\end{lemma}


Let $\left\{n_{k} \right\}_{k=1}^{\infty } $ be a sequence with $n_{k} \in \mathbb{N} \backslash \{ 1\} ,\, k \in
\mathbb{N}$. Let us recall that the  expression of $x$ in the following form
\[x=\sum_{k=1}^{\infty}\frac{\alpha_{k}}{n_{1} \cdot n_{2} \cdot \ldots \cdot n_{k}}=:\Delta_{\alpha_{1}\alpha_{2}...\alpha_{k}...},~\alpha_{k} \in \{0,~1,~...,~n_{k}-1\}\]
is said to be the Cantor series expansion of a real number $x \in [0, 1]$.


For a given sequence $\left\{n_{k} \right\}_{k=1}^{\infty }$ let $\mathcal{A}_{k}$ be the family of the k-th rank intervals (cylinders) , i.e.,
\[\mathcal{A}_{k}:=\{E:E=\Delta_{\alpha_{1}\alpha_{2}...\alpha_{k}}, ~\alpha_{i}\in\overline{1,n_{i}}, ~i=1,~2,~...,~k\}.\]
Let $\mathcal{A}$ be the family of all possible rank intervals, i.e.,
\[\mathcal{A}:=\{E:E=\Delta_{\alpha_{1}\alpha_{2}...\alpha_{k}}, ~n\in \mathbb{N}, ~\alpha_{i}\in\overline{1,n_{i}}, ~i=1,~2,~...,~n\}, \] which is said to be the \textit{Cantor covering family}.




The following theorem gives necessary and sufficient conditions for a Cantor covering family to be faithful.

\begin{theorem} The family $\mathcal{A}$ of Cantor coverings of the unit interval is faithful for the Hausdorff dimension calculation if and only if

\begin{equation} \label{main_equ}
\mathop{\lim }\limits_{k\to \infty } \frac{\ln n_{k} }{\ln n_{1} \cdot n_{2} \cdot \ldots \cdot n_{k-1} } =0.
\end{equation}

\end{theorem}

\begin{proof} \textbf{Sufficiency}. Let (\ref{main_equ}) holds. It is enough to prove that \[\dim_{H}(E) \geq \dim_{H}(E,\Phi),~ \forall E \subset [0,1].\]

Let $I$ be an arbitrary interval. Then there exists an interval $\Delta(k(I)) =\Delta_{\alpha _{1} ...\alpha _{k(I)} } \in \mathcal{A} $ such that:

1) $\Delta _{\alpha _{1} ...\alpha _{k(I)} } \subset I$;

2) any interval of $(k(I)-1)$ th rank does not belong to $I$.

\noindent The  interval $I$ contains at most $2\cdot n_{k(I)} $  intervals from $\mathcal{A}_{k}$. So $I$ can be covered by $N(\left|I\right|)=2\cdot n_{k(I)} +2$ intervals from $\mathcal{A}_{k}$. Therefore,

\[\left|\Delta(k(I)) \right|\le \left|I\right|< N(\left|I\right|)\cdot \left|\Delta(k(I)) \right| .\]

Let $C$ be an arbitrary positive constant. Then the equality
\[\mathop{\lim }\limits_{k\to \infty } \frac{\ln n_{k} }{\sum _{i=1}^{k-1}\ln n_{i}  } =0\]
holds if and only if  for any positive $\delta$ there exists $k_{0} (\delta)\in \mathbb{N}$ such that $\forall k>k_{0} (\delta): $
\[(2\cdot n_{k} +2) \cdot \left(\frac{2\cdot n_{k} +2}{n_{1} \cdot n_{2} \cdot ...\cdot n_{k-1} \cdot n_{k} } \right)^{\delta } \le C .\]

Therefore $\forall \delta>0,~\exists k_{0} (\delta ),~\forall k>k_{0} (\delta ):$
\[N(\left|I \right|)\cdot \left|I \right|^{\delta }  \le C.\]
So, from Lemma \ref{lem_trust} it follows that $\mathcal{A}$  is faithful for the Hausdorff dimension calculation.

\textbf{Necessity}. Now we show that if
\begin{equation} \label{GrindEQ__1_}
\mathop{\overline{\lim }}\limits_{k\to \infty } \frac{\ln n_{k} }{\ln n_{1} \cdot n_{2} \cdot \ldots \cdot n_{k-1} } =:C>0,
\end{equation}
\noindent then $\mathcal{A}$  is non-faithful for the Hausdorff dimension calculation. To this end we shall construct a set $T=T(C)$ with the following properties:

\[1) \dim _{H} (T)\le \frac{2}{2+C} ;\]

\[2) \dim _{H} (T,\mathcal{A} )\ge \frac{4+C}{4+3C} .\]

\noindent From (\ref{GrindEQ__1_}) it follows that there exists a subsequence $\{ k_{i} \} $ such that $\forall \delta \in (0,C),~ \exists N_{0} (\delta ),$ $\forall k_{i}>N_{0} (\delta)$:

\begin{equation} \label{GrindEQ__2_}
(n_{1} n_{2} \ldots n_{k_{i} -1} )^{C-\delta } \le n_{k_{i} } \le (n_{1} n_{2} \ldots n_{k_{i} -1} )^{C+\delta } .
\end{equation}

\noindent It is clear that $\forall \varepsilon >0,\, \, \exists \, N_{1} (\varepsilon )\, :\, \forall k>N_{1} (\varepsilon ):$
\begin{equation} \label{GrindEQ__3_}
\frac{1}{n_{1} \cdot n_{2} \cdot \ldots \cdot n_{k-1} } <\varepsilon .
\end{equation}
Let  $N_{2} (\varepsilon ,\delta ):=\max \{ N_{0} (\delta ),~N_{1} (\varepsilon )\, \} .$ Let us choose a subsequence $\left\{k'_{j} \right\}$ from the sequence $\{ k_{i} \} $ with the following property:
\begin{equation} \label{intrest _inequality}
 \frac{\ln(n_{k'_{j-1} +1} \ldots n_{k'_{j} -1})}{\ln (n_{1} n_{2} ...n_{k'_{j-1} -1} n_{k'_{j-1} } n_{k'_{j-1} +1} \ldots n_{k'_{j} -1})}>1-\frac{C}{4},
\end{equation}
\noindent and construct the set $T$ in the following way:

$$
T=
\bigg\{x: x\in[0,1),~ x=\sum _{k=1}^{\infty }\begin{array}{l} {\frac{\alpha _{k}(x) }{\prod _{i=1}^{k}n_{i}  } ,}~\alpha _{k} (x)\in \overline{0,\left[\sqrt{n_{k} } \right]\, } \end{array}
$$

$$
\textrm{if}~
\ k\in \{ k'_{j} \},~ \textrm{and}~ \alpha _{k}(x) \in \overline{0,n_{k} -1\, } ~ \textrm{if}~\, k\notin \{ k'_{j}\}\bigg\}.
$$

Firstly let us show that

\begin{equation} \label{first_inequality}
\dim _{H} (T)\le \frac{2}{2+C} .
\end{equation}

Let $k'_{j}>N_{2} (\varepsilon ,\delta )$. The set $T$ can be covered by $n_{1} \cdot n_{2} \cdot \ldots \cdot n_{k'_{j}-1} $ intervals and each of them is a union of $\left[\sqrt{n_{k'_{j}} } \right]+1$ sets from $\mathcal{A}_{k'_{j}}$. The $\alpha$-volume of this $\varepsilon$-covering is equal to

\[ n_{1} n_{2} \ldots n_{k'_{j}-1} \left(\frac{\left[\sqrt{n_{k'_{j}} } \right]+1}{n_{1} n_{2} \ldots n_{k'_{j}} } \right)^{\alpha }.\]

 From (\ref{GrindEQ__2_}) it follows that

\[\begin{array}{l} n_{1} n_{2} \ldots n_{k'_{j}-1} \left(\frac{\left[\sqrt{n_{k'_{j}} } \right]+1}{n_{1} n_{2} \ldots n_{k'_{j}} } \right)^{\alpha } \mathop{\le }2^{\alpha } (n_{1} n_{2} \ldots n_{k'_{j}-1} )^{1-\frac{1}{2} \alpha (C-\delta )-\alpha } , \end{array}\]

\noindent Suppose
\[1-\frac{1}{2} \alpha (C-\delta )-\alpha <0,\]

\noindent then
\[H_{\varepsilon }^{\alpha } (T)\le \mathop{\lim }\limits_{j\to \infty } 2^{\alpha } (n_{1} n_{2} ...n_{k'_{j}-1} )^{1-\frac{1}{2} \alpha (C-\delta )-\alpha } =0.\]
\noindent Therefore,

\[H_{\varepsilon }^{\alpha } (T)=0,~ \forall \alpha >\frac{2}{C-\delta +2} ,~ \forall \varepsilon >0,~ \forall \delta >0.\]
\noindent So,
\[\dim _{H} T\le \frac{2}{C-\delta +2},~\forall \delta >0,\]
and, hence,
\[\dim _{H} T\le \frac{2}{C +2}.\]

Now let us show that
\[\dim _{H} (T,\mathcal{A} )\ge \frac{4+C}{4+3C} .\]

Let
\[\{k''_{j}\}=\{k'_{j}\}\bigcap\{N_{2}(\varepsilon,\delta)+1,~N_{2}(\varepsilon,\delta)+2,~...\}.\]

\noindent Let $\mu=\mu_{N_{2}(\varepsilon,\delta)}$ be the probability measure corresponding to the random variable

\[\xi =\sum _{k=1}^{\infty }\frac{\xi _{k} }{\prod _{i=1}^{k}n_{i}  } ,\,  \]

\noindent where $\xi _{k} $ are independent random variables; if $k\in \{ k''_{j} \}$, then $\xi _{k} $ takes values $0,~1,~ ..., ~[\sqrt{n_{k} }]  $ with probabilities $\frac{1}{\left[\sqrt{n_{k} } \right]+1}$;
if $k\notin \{ k''_{j} \}$, then $\xi _{k} $ takes values $0,~1,~ ...,~ n_{k} -1 $ with probabilities $\frac{1}{n_{k}}$.

Then
\[\left|\Delta _{\alpha _{1} \alpha _{2} ...\alpha _{k} } \right|=\frac{1}{n_{1} n_{2} ...n_{k} } \]
for any $\Delta _{\alpha _{1} \alpha _{2} ...\alpha _{k} } $ from $\mathcal{A}_{k}$,
and
\[\mu (\Delta _{\alpha _{1} \alpha _{2} ...\alpha _{k} } )=\frac{1}{\varphi _{1} \varphi _{2} ...\varphi _{k} } \]
where  $\varphi _{t} =n_{t}$ if $t\notin \{ k''_{j} \} $ and $\varphi _{t} =\left[\sqrt{n_{t} } \right]+1$ if $t\in \{ k''_{j} \} $, $\forall t\in \mathbb{N}$.

Let us show that

\begin{equation} \label{GrindEQ__411_}
\frac{\ln (\mu (\Delta _{\alpha _{1} \alpha _{2} ...\alpha _{k} } ))}{\ln (\left|\Delta _{\alpha _{1} \alpha _{2} ...\alpha _{k} } \right|)} \ge \frac{4+C-2\delta }{4+3C+4\delta } ,\, \forall k\in \mathbb{N}.
\end{equation}

\noindent Taking into account  properties of $\{k''_{j}\}$, one can prove by induction on $j$ that
\begin{equation} \label{GrindEQ__42_}
\frac{\ln \left( \varphi _{1} \varphi _{2} ...\varphi _{k''_{j} } \right)}{\ln\left( n_{1} n_{2} n_{3} ...n_{k''_{j} }\right) } \ge \frac{4+C-2\delta }{4+3C+4\delta } ,\, \forall j\in \mathbb{N}.\,
\end{equation}

\noindent Let  $k \in(k''_{j},k''_{j+1})$. Then
\[\frac{\ln \left(\mu \left(\Delta _{\alpha _{1} \alpha _{2} ...\alpha _{k} } \right)\right)}{\ln \left(\left|\Delta _{\alpha _{1} \alpha _{2} ...\alpha _{k} } \right|\right)} \ge \frac{\ln \left(\mu \left(\Delta _{\alpha _{1} \alpha _{2} ...\alpha _{k''_{j} } } \right)\right)}{\ln \left(\left|\Delta _{\alpha _{1} \alpha _{2} ...\alpha _{k''_{j} } } \right|\right)} \ge \frac{4+C-2\delta }{4+3C+4\delta } ,\, \forall k\in \mathbb{N}. \]

Let $\{ \Delta' _{i} \}$ be an arbitrary  $\varepsilon $ - covering of Т,  $ \Delta' _{i}\in \mathcal{A},~\forall i \in \mathbb{N}$. Then, using (\ref{GrindEQ__42_}) we get
\noindent
\[\frac{4+C-2\delta }{4+3C+4\delta } \le \frac{\ln (\mu (\Delta' _{i} ))}{\ln (\left|\Delta' _{i} \right|)} <1\]

\noindent which implies that

\[\mu (\Delta' _{i} )\le \left|\Delta' _{i} \right|^{\frac{4+C-2\delta }{4+3C+4\delta } } .\]

Let
$ \alpha \in [0,\frac{4+C-2\delta }{4+3C+4\delta } ).$
Then we have
\[1=\mu (T)\le \bigcup _{i}\mu (\Delta' _{i} )\le \sum _{i}\left|\Delta' _{i} \right|^{\frac{4+C-2\delta }{4+3C+4\delta } } \le  \sum _{i}\left|\Delta' _{i} \right|^{\alpha }.\]
So,
$\forall \delta>0,~\forall \varepsilon>0,~ \alpha \in [0,\frac{4+C-2\delta }{4+3C+4\delta } )$, and for any $\varepsilon$ - covering  $\{ \Delta' _{i} \}$ of the set $T$ by cylinders $\Delta' _{i} \in \mathcal{A}$ we have  :
\[\sum _{i}\left|\Delta' _{i} \right|^{\alpha } \ge 1 . \]
Therefore,
\[H^{\alpha}_{\varepsilon}(T,\mathcal{A}) \geq 1,~~\forall \delta>0,~ \forall \varepsilon>0,~ \alpha \in [0,\frac{4+C-2\delta }{4+3C+4\delta } ).\]
So,
\[H^{\alpha } (T,\mathcal{A} )\ge 1, ~~~ \forall \delta >0,~ \forall \alpha <\frac{4+C-2\delta }{4+3C+4\delta }.\]
Hence,
\[\dim _{H} (T,\mathcal{A} )\ge \frac{4+C-2\delta }{4+3C+4\delta } ,\, \forall \delta >0,\]
and, therefore,
\[ \dim _{H} (T,\mathcal{A} )\ge \frac{4+C}{4+3C}, \]
which completes the proof.
\end{proof}

\section{Some applications}
First application of this theorem will be connected with fine fractal properties of random Cantor expansions.
Let us recall that for a given probability measure $\mu$ the number $$
  \dim_H \mu = \inf \{ \dim_H
  (E): ~ \mu(E)=1\}
$$
 is said to be the Hausdorff dimension of the measure $\mu$. In the case of singularity this number is a rather important  characteristic of a probability measure (see, e.g., \cite{AT2}).
Applying the latter theorem and methods from \cite{AT2}, we get the Hausdorff dimension of the probability distribution $\mu_{\xi}$ of the  random variable  $\xi$ with independent digits of  the Cantor series expansion, i.e.,  $$\xi=\sum_{k=1}^{\infty}\frac{\xi_{k}}{n_{1}n_{2}...n_{k}},$$ where independent $\xi_k$ take values $0,1,..., n_k-1$ with probabilities  $p_{0k}, p_{1k}, ... ,p_{n_{k}-1,k}$ respectively.

\begin{proposition}\label{H dim for rce}
 Let $h_j= - \sum\limits_{i=0}^{n_k-1} p_{ij} \ln p_{ij}$ and $
H_k = \sum\limits_{j=1}^k h_j.$
  If  \begin{equation}\label{condition for H dim}
      \sum\limits_{k=1}^{\infty}\left(\frac{\ln n_{k}}{\ln \left(n_{1}n_{2}...n_{k} \right)} \right)^2 <\infty,
    \end{equation}
     then the Hausdorff dimension of the probability distribution $\mu_{\xi}$ of the  random variable  $\xi$ with independent digits of  the Cantor series expansion  is equal to
 $$
\dim_H(\mu_{\xi})=\lim\limits_{\overline{ k \to
\infty}}\frac{H_k}{\ln \left( n_{1}n_{2}...n_{k} \right) }.$$
\end{proposition}

Now let us consider examples which show essential differences   between the notions  of \textit{faithful net-coverings} and \textit{comparable net-coverings}.

\textbf{Example 1.} Let $n_k=4^k,~\forall k \in \mathbb{N}$ and let $\mathcal{A}$ be  the net-covering family generated by the corresponding  Cantor series expansion. Let

$$
A=
\bigg\{x: x\in[0,1],~ x=\sum _{k=1}^{\infty }\begin{array}{l} {\frac{\alpha _{k}(x) }{\prod _{i=1}^{k} n_i } ,}~\alpha _{k} (x)\in \{0,1, ..., 2^k-1\},  ~~\forall k \in \mathbb{N} \end{array}\bigg\}.
$$
Then

$1) \dim_{H}A=\frac{1}{2};$

$2) H^{\frac{1}{2}}(A,\mathcal{A}) \geq 1;$

$3) H^{\frac{1}{2}}(A)=0.$

\begin{proof}

\noindent Let $\lambda$ be Lebesgue measure on the unit interval and let $\mu_{\xi}$ be the probability measure of the random variable

\[\xi =\sum _{k=1}^{\infty }\frac{\xi _{k} }{\prod _{i=1}^{k}n_{i}  } ,\,  \]

\noindent where $\xi _{k} $ are independent random variables taking values $0,~1,~ ..., ~2^k-1 $ with probabilities $\frac{1}{2^k}$. Let $\Delta_{n}(x)$ be the n-th rank cylinder of the Cantor series expansion containing $x$.  It is clear that for any $x \in A$ one has
$$\mu_{\xi}(\Delta_n(x))=2^{-\frac{n(n+1)}{2}}~~~ \mbox{and} ~~~ \lambda (\Delta_n(x))=4^{-\frac{n(n+1)}{2}}.$$
 So,
\begin{equation} \label{Bil_eq}
\frac{\ln \mu_{\xi}(\Delta_n(x))}{\ln \lambda (\Delta_n(x))}=\frac{1}{2}, ~~~ \forall x \in A.
\end{equation}

Using the Theorem 2.5 from  \cite{B}, we get $\dim_{H}(A,\mathcal{A})=\frac{1}{2}.$ From our theorem one can obviously derive the faithfulness of the family $\mathcal{A}$ for the case $n_k = 4^k$. Therefore $ \dim_{H}(A)=\frac{1}{2}$.

Let $\{ E _{j} \}$ be an arbitrary $\varepsilon $-covering of the set $A$ by cylinders from   $\mathcal{A}.$  Without loss of generality we may assume that $E_j \cap A \neq \emptyset$, i.e., $E_j = \Delta_{n_j}(x)$ for some $x \in A$.
Applying the mass distributional principle, we have
\[1=\mu (A)= \mu (\bigcup _{j} E _{j} )\le \sum _{j} \mu( E _{j}) = \sum _{j} |E_{j}|^{\frac{1}{2}}\]
for any $\varepsilon$-covering of  $A$ by cylinders from $\mathcal{A}$.
Therefore, $H^{\frac{1}{2}}(A,\mathcal{A}) \geq 1$.

The set $A$ can be covered by $2^1 \cdot 2^2 \cdot \ldots \cdot 2^{k-1} \cdot 1$ intervals (each of them is a union of $2^k$ k-th rank cylinders)  with length  $2^{-k^2}$. The $\frac{1}{2}$-volume of this covering is equal to
$2^{\frac{(k-1)k}{2}} \cdot \left( 2^{-k^2}\right)^{\frac{1}{2}},$  which tends to 0 as $k \rightarrow \infty$. Therefore,   $H^{\frac{1}{2}}(A)=0$.
\end{proof}

The following example shows that a faithful net-covering family can be "extremely non-comparable" to the Hausdorff measure.

\textbf{Example 2.} Let $n_k=4^k$ and let $\mathcal{A}$ be the net-covering family generated by the corresponding Cantor series expansion. Let
$$
T=
\bigg\{x: x\in[0,1],~ x=\sum _{k=1}^{\infty } \frac{\alpha _{k}(x) }{\prod _{i=1}^{k} 4^i } ~\text{with} ~\alpha _{k} (x)\in \overline{0, \sqrt{n_k}-1\, } ~\text{if}~ k \neq 2^s,
$$

$$
~\alpha _{k} (x)\in \overline{0,k \cdot \sqrt{n_k} -1\, } ~\text{if}~ k = 2^s , ~s \in \mathbb{N} \bigg\}.
$$

Then the family  $\mathcal{A}$ is faithful for the Hausdorff dimension calculation and

$1) \dim_{H}T=\frac{1}{2};$

$2) H^{\frac{1}{2}}(T,\mathcal{A}) = +\infty;$

$3) H^{\frac{1}{2}}(T)=0.$

\begin{proof}

\noindent Let $\mu_{\xi}$ be the probability measure with respect to the random variable

\[\xi =\sum _{k=1}^{\infty }\frac{\xi _{k} }{\prod _{i=1}^{k}n_{i}  } ,\,  \]

\noindent where $\xi _{k} $ are independent random variables with following distributions:

 if $k\neq2^s$, then $\xi _{k} $ takes values $0,~1,~ ..., ~2^k-1 $  with probabilities $\frac{1}{2^k}$;

if $k=2^s$, then $\xi _{k} $ takes values $0,~1,~ ...,~ k \cdot 2^k-1$ with probabilities $\frac{1}{k \cdot 2^k}$.

Let $\Delta_{n}(x)$ be the n-th rank cylinder of the Cantor series expansion containing $x$. From the construction of $\xi$ it follows that for any $x \in T$ one has  $$\mu_{\xi}(\Delta_n(x))=2^{-\left(\frac{n(n+1)}{2}+\frac{([\log_2 n ]+1)[\log_2 n]}{2}\right)}~~~ \mbox{and}~~~ \lambda (\Delta_n (x))=4^{-\frac{n(n+1)}{2}}.$$  So,
\begin{equation} \label{Bil_eq}
\lim\limits_{n \rightarrow \infty} \frac{\ln \mu_{\xi}(\Delta_n(x))}{\ln \lambda (\Delta_n(x))}=\frac{1}{2}, \forall x \in T.
\end{equation}

\noindent Using Theorem 2.5 from  \cite{B} and the faithfulness of $\mathcal{A} $ we get $$\dim_{H}(T,\mathcal{A})= \dim_{H}(T)= \frac{1}{2}.$$

For a given $m \in \mathbb{N}$ let us consider $2^m$  probability measures $\mu^j,~j=\overline{0,2^m-1}$  corresponding to the random variables
\[\xi^j =\sum _{k=1}^{\infty }\frac{\xi^j _{k} }{\prod _{i=1}^{k} 4^i  } ,\,  \]
whose independent digits $\xi^j _{k} $ have the following distributions:

 if $k \neq 2^s$, then $\xi^j _{k} $ takes values $0,~1,~ ..., ~2^k-1 $  with probabilities $\frac{1}{2^k}$;

if $k=2^s,~s \neq m$, then $\xi^j _{k} $ takes values $0,~1,~ ...,~ k \cdot 2^k-1$ with probabilities $\frac{1}{k \cdot 2^k}$;

 if $k=2^m$, then $\xi^j _{k} $ takes values $j \cdot  2^k+ 0,~j \cdot  2^k+1,~ ...,~ (j+1) \cdot 2^k-1$ with probabilities $\frac{1}{ 2^k}$.

Let $S_j$ be the spectrum of the measure $\mu^j$. From the construction of these measures and the definition of the set $T$ it follows that $S_j \bigcap S_i = \emptyset$ and $T = \bigcup_{i=1}^{2^m} S_i.$ Taking into account inequality $ \frac{\ln \mu^j(\Delta_n (x))}{\ln \lambda (\Delta_n(x))} \geq \frac{1}{2}, \forall x \in S_j$, and applying the mass distribution principle simultaneously for all measures $m^j$, we get $H^{\frac 12}(T, \mathcal{A})\geq 2^m$.
Since $m \in \mathbb{N}$ can be chosen arbitrarily, we have a desired conclusion about infiniteness of   $H^{\frac 12}(T, \mathcal{A})$.

On the other hand the set $T$ can be covered by $2^1 \cdot 2^2 \cdot \ldots \cdot 2^{2^s-1} \cdot 2^1 \cdot  2^2 \ldots 2^{s-1} \cdot 1 = 2^{\frac{(2^s-1)2^s}{2}+\frac{(s-1)s}{2}}$ intervals,  each of them is a union of $2^s2^{2^s}$  cylinders from $\mathcal{A}_{2^s}$ with length $\left(\frac{1}{4}\right)^{\frac{2^s(2^s-1)}{2}} \cdot \frac{2^s 2^{2^s}}{4^{2^s}} =
\left(\frac{1}{2}\right)^{2^{2s}-s} $. The $\frac{1}{2}$-volume of this covering is equal to
$$2^{\frac{(2^s-1)2^s}{2}+\frac{(s-1)s}{2}}
  \left(2^{-2^{2s}+s}\right)^{\frac{1}{2}}
  =2^{-\frac{1}{2}(2^{s}-s^2)} \rightarrow 0, ~~~ (s \rightarrow \infty).$$
 Therefore, $H^{\frac{1}{2}}(T)=0$.
\end{proof}

By using the same techniques it is not hard to prove the following result.

\begin{proposition}
Let $n_k=4^k$  and let $\mathcal{A}$ be the corresponding faithful net-covering family generated by the  Cantor series expansion. Then for any $\alpha \in (0,1)$ there exists a set $T_{\alpha}$ such that

$1) \dim_{H} T_{\alpha}=\alpha;$

$2) H^{\alpha}(T_{\alpha},\mathcal{A}) = + \infty;$

$3) H^{\alpha}(T_{\alpha})=0.$
\end{proposition}
\begin{proof}
  If $\alpha = \frac{p}{q} \in (0,1)$ is a rational number, then the proof is completely similar to those in example 2, but in the definition of the set $V_k$ of digits which are admissible at the k-th step of construction of $T_{\alpha}$, we define $V_k$ to be $\{0,1,..., \left[(n_k)^{\frac{p}{q}}\right]-1\}$ instead of $ \{0,1,..., \left[(n_k)^{\frac{1}{2}}\right]-1\}$ for $k\neq 2^s$, and $V_k = \{0,1,..., k \cdot \left[(n_k)^{\frac{p}{q}}\right]-1\}$ for $k= 2^s$, where $\left[ \cdot \right]$ denotes the integer part of an argument.

  If $\alpha  $ is an irrational number from $(0,1)$, then  we choose an increasing sequence $\{\frac{p_k}{q_k}\}$ of rational numbers converging to $\alpha$ and apply the same technics with the following definition of the set  $V_k := \{0,1,..., \left[(n_k)^{\frac{p_k}{q_k}}\right]-1\}$.
\end{proof}

\begin{remark}
The sequence $n_k$ has been chosen to be $\{4^k\}$ only for the simplicity of calculations in the above examples. In the forthcoming paper we shall show how the latter statement can be generalized and give necessary and sufficient conditions for the Cantor net-coverings to be comparable.
\end{remark}

\begin{remark}
  The latter proposition shows extreme differences between comparable and faithful net-coverings and demonstrates that the class of faithful net-coverings is  essentially wider then the class of comparable ones. The relation between these two classes is similar to the relation between bi-Lipshitz transformations and transformations preserving the Hausdorff dimension (see, e.g., \cite{APT, APT2} for details). More deep connections between faithfulness of net-coverings and the theory of transformations preserving the Hausdorff dimension will also be discussed in the forthcoming paper.
\end{remark}

\textbf{Acknowledgment}

This work was partly supported by DFG 436 UKR
113/97 project,  DFG KO 1989/6-1 project, and by the Alexander von Humboldt Foundation.


\bigskip
\begin{thebibliography}{999}




\bibitem{AlbeverioTorbin2007}
S. Albeverio, G. Torbin,
\textit{ On fine fractal properties of generalized infinite Bernoulli convolutions},
Bull. Sci. Math., \textbf{132}(2008), No. 8, P. 711--727.

\bibitem{AlbKonNikTor}
S. Albeverio, Yu. Kondratiev, R. Nikiforov, G. Torbin,
\textit{On fractal phenomena connected with infinite linear IFS and related singular  probability measures},
submitted to  \emph{Math. Proc. Cambridge Phil. Soc.}

\bibitem{AKPT2011}
S. Albeverio, V. Koshmanenko, M. Pratsiovytyi, G. Torbin,
 \textit{On fine structure of singularly continuous
probability measures and random variables with independent $\widetilde{Q}$ - symbols},
Methods of Functional Analysis and Topology, \textbf{17}(2011), No. 2, P. 97--111.



\bibitem{APT}
S. Albeverio, M. Pratsiovytyi, G. Torbin,
\textit{Fractal probability distributions and transformations preserving the
Hausdorff-Besicovitch dimension},
Ergodic Theory and Dynamical Systems, \textbf{24}(2004), No. 1, P. 1--16.



\bibitem{APT2}
S. Albeverio, M. Pratsiovytyi, G. Torbin,
\textit{Transformations preserving the Hausddorff-Besicovitch dimension},
Central European Journal of Mathematics, \textbf{6}(2008), No. 1, P. 119--128.


\bibitem{AT2}
S. Albeverio, G. Torbin,
\textit{Fractal properties of singularly
continuous probability distributions with independent $Q^{*}$-digits},
Bull. Sci. Math., \textbf{129}(2005),  No. 4, P. 356--367.

\bibitem{FG}
C. Bandt, S. Graf, M. Z$\ddot{a}$hle - eds,
\textit{Fractal Geometry and stochastics},
Basel, Boston, Berlin, Birkh$\ddot{a}$user, (2000).

\bibitem{Baranski}
K. Baranski,
\textit{Hausdorff dimension of the limit sets of some planar
geometric constructions},
Advances in Mathematics, \textbf{210}(2007), P. 215--245.

\bibitem{BernardiBondioli}
M. Bernardi, C. Bondioli,
\textit{On some dimension problems for self-affine fractals},
Journal for  Analysis and its Applications, \textbf{18}(1999), No. 3, P. 733--751.


\bibitem{Besicovitch}
A. Besicovitch,
\textit{On existence of subsets of finite measure
of sets of infinite measure},
Indag. Math, \textbf{14}(1952), P. 339--344.


 \bibitem{B}
 P. Billingsley,
 \textit{Hausdorff dimension in probability theory II},
 Ill. J. Math., (1961), No. 5, P. 291--198.

\bibitem{Bil1}
P. Billingsley,
\textit{ Ergodic theory and information},
New York-London-Sydney: John Wiley \& Sons, Inc., (1965).

\bibitem{Cantor1869}
G. Cantor,
\textit{$\ddot{U}$ber die einfachen Zahlensysteme},
Zeitschrift f. Math. u. Physik.,
\textbf{14}(1869), P. 121--128.

\bibitem{Cooper}
M. Cooper,
\textit{ Dimension, measure and infinite Bernoulli convolutions},
Math. Proc. Cambr. Phil. Soc., \textbf{124}(1998), P. 135--149.

\bibitem{Cutler}
C. Cutler,
\textit{A note on equivalent interval covering systems for Hausdorff dimension on $R$},
Internat. J. Math. and Math. Sci., \textbf{2}(1988), No. 4, P. 643--650.



\bibitem{ER_1959}
P. Erd\"os, A. Renyi,
\textit{Some further statistical properties of the digits in Cantor's series},
Acta Math. Acad. Sci. Hungar., \textbf{10}(1959), P. 207--215.


\bibitem{Everett}
C. Everett,
\textit{Representations for real numbers},
Bull. Amer. math. Soc., \textbf{52}(1946), P. 861--869.

\bibitem{Falconer85}
K. Falconer,
\textit{The geometry of fractal sets},
Cambridge University Press, (2002).



\bibitem{Fal}
K. Falconer,
\textit{Fractal Geometry: Mathematical Foundations and Applications},
Chichester: Wiley, (1990).

\bibitem{Fal97}
K. Falconer,
\textit{Techniques in Fractal Geometry},
Chichester: Wiley, (1997).


\bibitem{G_1976}
J. Galambos,
 \textit{Representations of real numbers by infinite series},
Lecture Notes in Mathematics,Springer-Verlag, Berlin-New York, \textbf{502}(1976).


\bibitem {Hutchinson2}
J. Hutchinson,
\textit{Fractals and self similarity},
Indiana Univ. Math. J., \textbf{30}(1981), P. 713--747.

\bibitem{Mandelbrot83}
B. Mandelbrot,
\textit{The Fractal geometry of nature},
Freeman and Co, San-Francisco, (1983).

\bibitem{ManceDiss}
B. Mance,
\textit{Normal numbers with respect to the Cantor series expansion},
Thesis (Ph.D.), The Ohio State University (2010).



\bibitem{Marstrand1954}
J. Marstrand,
\textit{The dimension of Cartesian product sets},
Proc. Cambridge Phil. Soc., \textbf{50}(1954), P. 198--202.


\bibitem{Mattila}
P. Mattila,
\textit{Geometry of sets and measures in euclidean spaces},
Cambridge University Press, (2004).

\bibitem{Peres}
Y. Peres, W. Schlag, B. Solomyak,
\textit{Sixty years of Bernoulli convolutions, Fractal geometry and stochastics II (Greifswald/Koserow,
1998)},
Progr. Probab., Birkh\"{a}user, Basel, \textbf{46}(2000). P. 39--65.

\bibitem{PeresSol}
Y. Peres, B. Solomyak,
\textit{Absolute continuity of Bernoulli convolutions, a simple proof},
Math. Res. Lett., \textbf{3}(1996), No. 2, P. 231--239.

\bibitem{PerTor}
Y. Peres, G. Torbin,
\textit{Continued fractions and dimensional gaps},
in preparation.

\bibitem{Pesin97}
Y. Pesin,
\textit{Dimension theory in dynamical systems.
Contemporary views and applications},
Chicago Lectures in Mathematics. University of Chicago Press, Chicago, (1997).


\bibitem{P}
M. Pratsiovytyi,
\textit{Fractal approach to
investigations of singular distributions},
National Pedagogical Univ., (1998).


\bibitem{PT_NZ2003}
M. Pratsiovytyi, G. Torbin,
\textit{On analytic (symbolic) representation of one-dimensional continuous transformations preserving the Hausdorff-Besicovitch dimension},
Transactions of the National Pedagogical University of Ukraine.  Mathematics, \textbf{4}(2003), P. 207--215.

\bibitem{Renyi}
A. Renyi,
\textit{Representations for real numbers and their ergodic properties},
Acta Math. Sci. Hungar., \textbf{8}(1957), P. 477--493.

\bibitem{R}
C. Rogers,
\textit{Hausdorff measures},
Cambridge Univ. Press, London, (1970).


\bibitem{Schweiger}
F. Schweiger,
\textit{Ergodic Theory of Fibred Systems and Metric
Number Theory},
 Oxford: Clarendon Press, (1995).

\bibitem{Shir}
A. Shiryaev,
\textit{Probability},
Springer-Verlag, New York, (1996).

\bibitem{Solomyak95}
B. Solomyak,
\textit{On the random series $\sum \pm
\lambda ^n$ (an Erd\"{o}s problem)},
Ann. of Math., \textbf{142}(1995), No. 3, 611--625.


\bibitem{TuP}
A. Turbin, M. Pratsiovytyi,
\textit{Fractal sets, functions, and distributions},.
Kiev: Naukova Dumka, (1992).

\bibitem{Triebel}
H. Triebel,
\textit{Fractals and Spectra},
Basel, Boston, Berlin, Birkh\"{a}user, (1997).


\end{thebibliography}
\end{document}